\documentclass[11pt]{amsart}
\usepackage{amssymb}
\usepackage{latexsym}
\usepackage{amsmath}
\usepackage{amsthm}
\usepackage{amsfonts}
\usepackage{color}
\usepackage{hyperref}
\usepackage{pdfsync}
\usepackage[usenames,dvipsnames]{pstricks}
\usepackage{epsfig}
\usepackage{pst-grad} 
\usepackage{pst-plot} 

\newcommand{\R}{\mathbb{R}}

\newcommand{\C}{\mathcal{C}}

\newcommand{\I}{\mathcal{I}}

\newcommand{\Tr}{\mathrm{Tr}}

\newcommand{\osc}{\mathrm{osc}}

\theoremstyle{plain}
\newtheorem{defi}{Definition}[section]
\newtheorem{prop}[defi]{Proposition}
\newtheorem{teo}[defi]{Theorem}

\newtheorem{lema}[defi]{Lemma}

\newtheorem{remark}[defi]{Remark}

\theoremstyle{definition}

\theoremstyle{remark}

\numberwithin{equation}{section}

\begin{document}

\title[]{Interior regularity results for fractional elliptic equations that degenerate with the gradient}

\author[]{Disson dos Prazeres}
\address{Disson dos Prazeres:
Departamento de Matem\'atica , Universidade Federal de Sergipe,
Sala 19, São Cristóvão-SE, BRAZIL.
\newline {\tt disson@mat.ufs.br}
}

\author[]{Erwin Topp}
\address{
Erwin Topp:
Departamento de Matem\'atica y C.C., Universidad de Santiago de Chile,
Casilla 307, Santiago, CHILE.
\newline {\tt erwin.topp@usach.cl}
}

\date{\today}

\begin{abstract} 
	In this paper we obtain interior regularity estimates for viscosity solutions of nonlocal Dirichlet problems that degenerate when the gradient of the solution vanishes. Interior H\"older estimates are obtained when the order of the fractional diffusion is less or equal than one, and Lipschitz estimates when it is bigger than one. In the latter case, the estimates are robust enough to conclude interior $C^{1, \alpha}$ regularity by an improvement of the flatness procedure, which is possible when the nonlocal term is close enough to a second-order diffusion.
\end{abstract}

\maketitle


\section{Introduction}\label{secIntroduction}

In this paper we study interior regularity estimates for viscosity solutions $u$ to nonlinear elliptic problems with the form
\begin{equation}\label{maineq}
-|Du(x)|^\gamma \I(u,x) = f(x) \quad \mbox{for} \ x \in B_1,
\end{equation}
where $\gamma > 0$, $f \in L^\infty(B_1)$, $Du(x)$ stands for the gradient of the unknown function $u$ at $x$, and $\I(u, x)$ is a nonlocal operator, uniformly elliptic in the sense of Caffarelli and Silvestre~\cite{CS1,CS2}. To be more precise, we consider $\sigma \in (0,2)$ and $0 < \lambda \leq \Lambda < \infty$ and a family of symmetric kernels $\mathcal K_0$, given by measurable functions $K: \R^N \setminus \{ 0 \} \to \R$ such that
\begin{equation}\label{eliptico}
\lambda \frac{C_{\sigma,N}}{|x|^{N + \sigma}} \leq K(x) \leq \Lambda  \frac{C_{\sigma, N}}{|x|^{N + \sigma}}, \quad x \neq 0,
\end{equation}
where $C_{\sigma, N} > 0$ is a normalizing constant to be specified later.
For each $K \in \mathcal K_0$  and a function $u: \R^N \to \R$, we denote
\begin{equation}\label{operadorlineal}
I_K(u,x) = \mathrm{P.V.} \int_{\R^N} [u(y) - u(x)]K(x - y)dy,
\end{equation}
where $\mathrm{P.V.}$ stands for the Cauchy principal value of the integral. 

Notice that for each $K$, $I_K$ is a linear operator and it is defined for functions $u$ sufficiently smooth at $x$ and satisfying certain growth conditions at infinity, namely
\begin{equation}\label{L1sigma}
\| u \|_{L^1_\sigma(\R^N)} := \int_{\R^N}\frac{|u|}{1+|y|^\sigma}dy < +\infty.
\end{equation}

We say that such a function belongs to the class $L^1_\sigma(\R^N)$, and in particular, bounded functions belong to this class.

In the case $K(x) = C_{N, \sigma}|x|^{-(N + \sigma)}$, the normalizing constant $C_{N, \sigma}$ is taken in such a way $I_K = \Delta^{\sigma/2}$, the fractional Laplacian of order $\sigma$, and has the property that $\Delta^{\sigma/2} \to \Delta$ as $\sigma \to 2^-$ in an adequate functional framework, see~\cite{Hitch}. This stability property is going to be crucial in this note.

Then, for a two-parameter family of kernels $\{ K_{ij} \}_{ij} \subseteq \mathcal K_0$, the nonlinear operator $\I$ takes the form
\begin{equation}\label{operador}
\I(u,x) = \inf_{i} \sup_j I_{K_{ij}}(u,x).
\end{equation}

The main novelty here is the presence of the gradient term $|Du|^\gamma$ in~\eqref{maineq}, since the problem becomes degerate as soon as the gradient of the solution vanishes. Roughly, the information coming from the equation is missed when the gradient is null.

Equation~\eqref{maineq} is the nonlocal analog to second-order, fully nonlinear equations with the form
\begin{equation}\label{tangenteq}
-|Du|^\gamma F(D^2u) = f \quad \mbox{in} \ B_1,
\end{equation}
where $F: \mathbb S^{N} \to \R$ is a nonlinear operator acting on the set of symmetric matrices $\mathbb S^{N}$, uniformly elliptic in the sense that there exists constants $0 < \lambda \leq \Lambda < +\infty$ such that
$$
\lambda \Tr(N) \leq F(M + N) - F(M) \leq \Lambda \Tr(N), \quad \mbox{for all} \ M, N \in \mathbb S^N, \ N \geq 0,
$$
where the last inequality is understood in the sense of matrices.
This equation appears as a fully nonlinear version of equations that degenerate with the gradient, showing a similar degeneracy behavior compared with more studied equations involving, for example, the $p$-Laplace operator in divergence form, for which regularity results are addressed in the standard weak formulation in Sobolev spaces. 

Due to the fully nonlinear nature of the problem~\eqref{tangenteq}, the viscosity solution's framework is a consistent notion of weak solution to address this problem, see~\cite{usersguide}. 
In this setting, interior regularity results for~\eqref{tangenteq} and similar problems have been addressed by several authors, see for instance~\cite{BD1,BD2,I, IS,DFQ, AR} and references therein. 
In the first part of this paper, we basically follow the procedure of Birindelli and Demengel developed in a series of papers, see for instance~\cite{BD1, BD2, BD3}, where the authors obtain H\"older and Lipschitz estimates for problems with the degenerate elliptic structure of~\eqref{tangenteq} exploiting the nowadays well-known method of Ishii-Lions~\cite{IL}. In the context of viscosity solutions for nonlocal problems presented in~\cite{BI}, and adapting to our degenerate case the Ishii-Lions method for nonlocal problems used in~\cite{BChI}, we are able to conclude the following interior estimates for~\eqref{maineq}.
\begin{teo}\label{first result}
	Let $f \in L^\infty(B_1)$ and $\I$ an operator with the form~\eqref{operador}. Let $u \in L^\infty_{loc} \cap L^1_\sigma$ be a viscosity solution to problem~\eqref{maineq}. Then,
	\begin{itemize}
		\item If $0<\sigma<1$ then $u\in C^{\sigma}$ and 
		$$
		[u]_{C^{\sigma}(B_{1/2})}\leq C (\|u\|_{L^1_\sigma(\R^N)}+\|u\|_{L^\infty(B_1)}+\|f\|_{L^\infty(B_1)}).
		$$
		
		\item If $\sigma = 1$, then $u \in C^{\alpha}$ for every $\alpha \in (0,1)$, and
		$$
		[u]_{C^{\alpha}(B_{1/2})}\leq C_\alpha (\|u\|_{L^1_\sigma(\R^N)}+\|u\|_{L^\infty(B_1)}+\|f\|_{L^\infty(B_1)}).
		$$

		\item If $1<\sigma<2$ then $u\in C^{0,1}$ and
		$$
		[u]_{C^{0,1}(B_{1/2})}\leq C_0(\|u\|_{L^1_\sigma(\R^N)}+\|u\|_{L^\infty(B_1)}+\|f\|_{L^\infty(B_1)}),
		$$
		and the constant $C_0$ is uniformly bounded as $\sigma \to 2^-$.
	\end{itemize}
\end{teo}

As we mentioned, the proof of this result is a consequence of Ishii-Lions method and elliptic estimates ``in the direction of the gradient" already proven by Barles, Chasseigne and Imbert in~\cite{BChI} (for H\"older estimates), and the same authors together with Ciomaga in~\cite{BCCI-Lip} (for Lipschitz estimates). Roughly speaking, the ``doubling variables" method in the Ishii-Lions argument determines a test function with a large gradient, which evaluated on the equation strenghten the ellipticity of the diffusion. 


\medskip

More interesting is the analysis of higher-order regularity. Both equations~\eqref{maineq} and~\eqref{tangenteq} show a strong nonlinear structure that prevents a ``linearization" mechanism and a  bootstrap procedure that allows to get regularity estimates for the gradient of their solutions. 
Thus, a different approach is necessary.
Related to this, Imbert and Silvestre in~\cite{IS} obtain interior $C^{1,\alpha}$ estimates for solutions of~\eqref{tangenteq} by an improvement of flatness of the solutions. The strategy is based on the study of equicontinuity properties for the auxiliar equation
\begin{equation*}
-|p + Du|^\gamma F(D^2 u) = f \quad \mbox{in} \ B_1,
\end{equation*}
in terms of $p \in \R^N$. Such an equation appears when it is measured how far is the graph of the solution from an hiperplane around each interior point. Using an itearive method relying on the stability property of viscosity solutions, the authors extract the regularity from the limit equation
\begin{equation}\label{hola1}
-|Du|^\gamma F(D^2 u) = 0 \quad \mbox{in} \ B_1,
\end{equation}
showing the equivalence among it and the equation
\begin{equation}\label{hola2}
F(D^2 u) = 0 \quad \mbox{in} \ B_1,
\end{equation}
for which it is known that interior $C^{1, \alpha}$ estimates hold, see Caffarelli and Cabr\'e~\cite{CC}.

\medskip

The proof of the equivalence among~\eqref{hola1} and~\eqref{hola2} in~\cite{IS} is intrinsically local, and in fact it is not true in the nonlocal framework, as the following example shows: consider the function $u: \R \to \R$ defined as $u(x) = x+1$ if $x \leq - 1$, $u(x) = 0$ if $x \in (-1,1)$, and $u(x) = x - 1$ if $x \geq 1$. For every $\sigma \in (1,2)$ and $\gamma > 0$, it is easy to see that $u$ satisfies
$$
-|u_x|^\gamma \Delta^{\sigma/2} u = 0 \quad \mbox{in} \ (-1,1), 
$$ 
(the fact that $u$ is unbounded here is not important since it belongs to the class $L^1_\sigma$), but we have that $x \mapsto \Delta^{\sigma}u(x)$ is a continuous function in $(-1,1)$ such that $\Delta^{\sigma}u(x) \to -\infty$ as $x \to -1^+$ and $\Delta^{\sigma}u(x) \to +\infty$ as $x \to 1^-$.
Despite this example does not imply the lack of interior $C^{1,\alpha}$ estimates for~\eqref{maineq}, it is illustrative about how the nonlocality combined with the gradient degeneracy create difficulties that do not arise in the second-order context. 

\medskip

Nevertheless, we provide a positive answer for interior $C^{1, \alpha}$ estimates for~\eqref{maineq} when the order of the nonlocal term $\I$ is sufficiently close to $2$, following the approximation prodecure described in~\cite{IS} in the spirit of section 8 of~\cite{CC}. We use the fact that if $\sigma$ is close to 2, the nonlocal problem~\eqref{maineq} approximate a second-order equation for which interior $C^{1, \alpha}$ estimates are at hand. However, in this procedure we require some extra continuity assumptions on the kernels $K$ definining $\I$ in order to identify the local limit associated to~\eqref{maineq}. 
\begin{teo}\label{second result}
	Let $f \in L^\infty(B_1)$ and $\I$ as in~\eqref{operador} defined through a family of kernels $\{ K_{ij} \}_{ij} \subset \mathcal K_0$, additionally satisfying the following property: there exist a modulus of continuity $\omega$ and a set $\{ k_{ij} \}_{i,j} \subset (\lambda, \Lambda)$ such that
	\begin{equation}\label{contK}
	|K_{ij}(x) |x|^{N + \sigma} - k_{ij}| \leq \omega(|x|), \quad |x| \leq 1.
	\end{equation}
	
	Then, there exists $\sigma_0 \in (1,2)$ close enough to $2$ such that for $\sigma_0<\sigma<2$ every bounded viscosity solution $u$ to~\eqref{maineq} is in $C^{1, \alpha}$ for some $\alpha \in (0,1)$, and
	$$
	[u]_{C^{1,\alpha}(B_{1/2})}\leq C_0(\|u\|_\infty+\|f\|_\infty^{\frac{\sigma-1}{1+\gamma}}).
	$$
\end{teo}

As we mentioned above, the proof of this result is performed by approximation to the second-order equation associated to~\eqref{maineq} as $\sigma \to 2^-$. By assumption~\eqref{contK}, this limiting problem takes the form
\begin{equation}\label{limIsaacs}
-|Du|^\gamma \inf_{i} \sup_{j} \Big{(} k_{ij} \Tr(D^2 u) \Big{)} = 0 \quad \mbox{in} \ B_1,
\end{equation}
for which, as we mentioned above, $C^{1, \alpha}$ estimates are proven in~\cite{IS}. In fact they show that solutions of (\ref{limIsaacs}) are solutions of 
$$
\inf_{i} \sup_{j} \Big{(} k_{ij} \Tr(D^2 u) \Big{)} = 0 \quad \mbox{in} \ B_1.
$$
Moreover, solutions of 
\begin{equation}
-|Du|^\gamma \inf_{i} \sup_{j} \Big{(} k_{ij} \Tr(D^2 u) \Big{)} = f(x) \quad \mbox{in} \ B_1,
\end{equation}
when $f\in L^\infty(B_1)$ are also $C^{1,\alpha}$. And in this case they ensure that this is the best regularity theory that you can expect even when $f$ is a constant and the second order operator is the Laplacian. The same phenomenon happens with the nonlocal case, since that $u=|x|^{1+\frac{\sigma-1}{1+\gamma}}$ is a solution of
$$
	-|Du|^\gamma (-\triangle^{\sigma/2}) = C \quad \mbox{in} \ B_1.
$$
In the spirit of \cite{ART} we show that in our case
$$
	\alpha<\min\left\lbrace\bar{\alpha},\frac{\sigma-1}{1+\gamma}\right\rbrace
$$
where $\bar{\alpha}>0$ is the regularity of the gradient of the solutions of~\eqref{limIsaacs}.

\medskip

Concerning existence and uniqueness, we can start mentioning that comparison principle fails for~\eqref{maineq}. We illustrate this in the homogeneous case $f \equiv 0$: consider a smooth nondecreasing function $\varphi:[0,+\infty) \to \R$ such that $\varphi(t) = 0$ for $t \in [0,1]$ and $\varphi(t) = 1$ for $t \geq 2$, and let $v(x) = \varphi(|x|)$. It is clear that $v$ is a (super)solution to~\eqref{maineq}. For $\epsilon \in (0,1)$, consider a nonnegative smooth function $\eta_\epsilon$ with support in the ball of radius $\epsilon$, such that $\eta_\epsilon(0) > 0$ and $||\eta_\epsilon||_{C^2(\R^N)} \to 0$ as $\epsilon \to 0^+$. Then, the function $u := v + \eta_\epsilon$ is a subsolution to~\eqref{maineq} for $\epsilon$ small enough, $u = v$ in $B_1^c$, but $u(0) > v(0)$. This is in contrast with the second-order counterpart of the problem (namely,~\eqref{hola1}) as it is shown by Barles and Busca in~\cite{BB}. This prevents the application of standard Perron's method to conclude existence of continuous viscosity solutions. 
Moreover, comparison principles have been used in other contexts to conclude equivalence among viscosity and other type of weak solutions, that a posteriori drives to conclusions about regularity, see for instance~\cite{JLM}. In view of the discussion above, this is not possible in our case.
Nevertheless, our regularity estimates in Theorem~\ref{first result} are robust enough to get compactness of the family of solutions of the ``vanishing viscosity" problem
	$$
	-\epsilon \Delta^{\sigma/2} u - |Du|^\gamma \I(u) = f \quad \mbox{in} \ B_1,
	$$
	in order to get existence by stability as $\epsilon \to 0^+$.


\medskip

Finally, we mention that our results can be easily extendable to other type of nonlocal equations involving nonsymmetric kernels, and lower order terms, but we prefer to state the results in this setting for simplicity.



\bigskip

\noindent
{\bf Notion of solution and organization of the paper.}
We introduce the notion of solution we consider for~\eqref{maineq}. We introduce some notation, for a bounded function $u$, a smooth function $\varphi$, and a measurable set $A \subseteq \R^N$ we write
\begin{equation*}
I_K[A](u, \varphi, x) = C_\sigma \mathrm{P.V.} \int_{A} (u(x + z) - u(x)) K(z)dz.
\end{equation*}

In order to introduce the notion of viscosity solution we are going to use here, for $p \in \R^N$ fixed we consider the auxiliar problem
\begin{equation}\label{eqp}
-|Du + p|^{\gamma} \I(u) = f \quad \mbox{in} \ B_1.
\end{equation}

Of course, our main equation~\eqref{maineq} is included in~\eqref{eqp} with $p = 0$.

\begin{defi}
	An upper semicontinuous function $u: \R^N \to \R$ is a viscosity subsolution to~\eqref{eqp} if for any $x \in B_1$ and every function $\varphi \in C^2(\R^N)$ such that $u - \varphi$ attains a local maximum at $x$ and $D\varphi(x) \neq -p$, then
	$$
	-|D\varphi(x) + p|^\gamma \I_{\delta}(u, \varphi, x) \leq f(x),
	$$
	where we have denoted
	\begin{equation}\label{eval}
	\I_{\delta}(u, \varphi, x) = \inf_i \sup_j \Big{(} I_{K}[B_\delta](\varphi, D\varphi(x), x) +  I_{K}[B_\delta^c](u, D\varphi(x), x) \Big{)}.
	\end{equation}
	
	In an analogous way it is defined supersolution and solution to the problem.
\end{defi}

\medskip

The paper is organized as follows: in Section~\ref{teo1} we provide the proof of Theorem~\ref{first result}, which is proven in a slightly more general form in order to get Theorem~\ref{second result} in the subsequent Section~\ref{teo2}.

\section{Proof of Theorem~\ref{first result}.}\label{teo1}

In this section we provide the proof of Theorem~\ref{first result}, stating a slightly stronger result in the case $\sigma > 1$ which is going to be useful in the proof of Theorem~\ref{second result}.



For $p \in \R^N$, we recall the auxiliar problem~\eqref{eqp} and obtain H\"older estimates for their solutions independent of $p$ through the following
\begin{prop}\label{propLipunif}
Assume $\sigma \in (1,2)$. For every $p \in \R^N$, each bounded viscosity solution $u$ to problem~\eqref{eqp} is Lipschitz continuous, and the Lipschtz constant depends only on $|u|_\infty, |f|_\infty$ and the data, but not on $p$.
\end{prop}

This result is a direct consequence of the following two lemmas.
\begin{lema}\label{lemaplarge}
	Assume $\sigma \in (1,2)$. There exists $c_0 > 1$ such that if $|p| \geq c_0$, each bounded viscosity solution to~\eqref{eqp} is Lipschitz continuous, with Lipschitz constant independent of $p$.
\end{lema}

\begin{proof}
	Scalling in terms of $|p|$, we consider the equivalent problem
	\begin{equation}
	-|\hat p + a_0 Du|^\gamma \I(u) = a_0^\gamma f,
	\end{equation}
	where $\hat p = p/|p|$ and $a_0 = |p|^{-1}$.
	
	\medskip
	
	Now, consider a nonnegative, smooth function $\tilde \psi: \R^N \to \R$ such that $\tilde \psi = 0$ in $B_{1/2}$ and $\tilde \psi = 1$ in $B_{3/4}^c$, and denote $\psi = (\osc_{\R^N}(u) + 1) \tilde \psi$. We also consider $\alpha \in (0,1)$ to be fixed small enough, and the function $\varphi: [0,+\infty) \to \R$ defined as $\varphi(t) = t - t^{1 + \alpha}$ for $t \in [0,t_0]$ and $\varphi(t) = \varphi(t_0)$ for all $t > t_0$, where $t_0 > 0$ is chosen sufficiently small in order to have $\varphi > 0$ in $(0,+\infty)$. Notice that $\varphi \in C^2(0,t_0) \cap C^1[0,t_0)$.
	
	We will prove that there exists a universal constant $\bar C > 0$ such that, taking 
	\begin{align}\label{choiceL}
	L = \bar C (|f|_\infty + \osc_{B_1}(u) + 1), 
	\end{align}
	then $u$ is Lipschitz continuous with Lipschitz constant $L$.
	
	We consider the function 
	$$
	\phi(x,y) := L \varphi(|x - y|) + \psi(y)
	$$ 
	and
	\begin{equation}\label{Phi}
	\Phi(x,y) = u(x) - u(y) - \phi(x,y), \quad x, y \in \R^N.
	\end{equation}
	
	By continuity, $\Phi$ attains its maximum in $\bar B_1 \times \bar B_1$ at $\bar x, \bar y$. If this maximum is nonpositive, then we have the interior Lipschitz result. By contradiction, we assume that the maximum is strictly positive. 
	
	In particular, by taking we have $L\varphi(x - y) \leq \osc_{B_1}(u)$, and therefore, by taking $L$ large in terms of $\osc(u)$, by the Lipschitz continuity of $\varphi$ we conclude that
	\begin{equation*}
	L | \bar x - \bar y| \leq \osc_{B_1}(u).
	\end{equation*}
	
	On the other hand, by the definition of $\psi$, we have $\bar y \in B_{3/4}$, and by taking $L$ larger (through $\bar C$ in~\eqref{choiceL}) if it is necessary, we conclude that $\bar x \in B_{7/8}$. Finally, note that $\bar x \neq \bar y$.
	
	Then, we write
	\begin{equation}\label{xis}
	\xi_1 = L\varphi'(|\bar x - \bar y|) \widehat{\bar x - \bar y}; \quad \xi_2 = \xi_1 - D\psi(\bar y),
	\end{equation}
	and for each $\delta \in (0,1)$, we can use the viscosity inequalities for $u$ at $\bar x$ and $\bar y$ to write
	\begin{align*}
	|\hat p + a_0 \xi_1|^\gamma \I_{\delta}(u, \phi(\cdot, \bar y), \bar x) & \geq a_0^\gamma f(\bar x), \\
	|\hat p + a_0 \xi_2|^\gamma \I_{\delta}(u, -\phi(\bar x, \cdot), \bar y) & \leq a_0^\gamma f(\bar y),
	\end{align*}
	where we have used the notation introduced in~\eqref{eval}.

	Then, we take $a_0 < 1$ small enough in terms of $L$ (which is going to be fixed in terms of $\osc (u)$ and the data, but not on $a_0$) to conclude that
	\begin{equation*}
	|\hat p + a_0 \xi_1|, |\hat p + a_0 \xi_2| \geq 1/2,
	\end{equation*}
	and dividing by the gradient, we conclude that
	\begin{align*}
	\I_{\delta}(u, \phi(\cdot, \bar y), \bar x) & \geq - C |f|_\infty, \\
	\I_{\delta}(u, -\phi(\bar x, \cdot), \bar y) & \leq C |f|_\infty,
	\end{align*}
	for some universal constant $C > 0$. Then, we substract these inequalities to get
	\begin{equation*}
	-2 C |f|_\infty \leq \I_{\delta}(u, \phi(\cdot, \bar y), \bar x) - \I_{\delta}(u, -\phi(\bar x, \cdot), \bar y),
	\end{equation*}
	and from now on we concentrate in the right-hand side. At this point, we follow the arguments presented in~\cite{BChI, BCCI-Lip} and that are nowadays well-known.
	
	
	For all $|z| \leq 1/8$, we have $\bar x + z, \bar y + z \in \bar B_1$. Then, there exists a kernel $K$ in the family such that
	\begin{equation}\label{Roma}
	-C(|f|_\infty + \osc_{B_1} (u) + \|u\|_{L^1_\sigma}) - 1 \leq I_1 + I_2, 
	\end{equation}
	where 
	\begin{align*}
	I_1 & := I_K[B_{\delta}](\phi(\cdot, \bar y), \bar x) - I_K[B_{\delta}](-\phi(\bar x, \cdot), \bar y), \\
	I_2 & := I_K[B_{1/8} \setminus B_\delta](u, \xi_1, \bar x) - I_K[B_{1/8} \setminus B_\delta](u, \xi_2, \bar y),
	\end{align*}
	and the constant $C > 0$ depends on the ellipticity constants.
	
	At this point, denote $e = \bar x - \bar y$, $\hat e = e/|e|$ and 
	\begin{align}\label{C}
	\C = \{ z \in B_\rho : |\langle \hat e, z \rangle| \geq (1 - \eta) |z|\},
	\end{align}
	for constants $\eta \in (0,1)$ and $\rho \in (0,1/8))$ to be fixed. 
	
	Using the maximality of $(\bar x, \bar y)$, it is possible to get that
	\begin{align*}
	I_K[\C \setminus B_\delta](u, \xi_1, \bar x) & \leq I_K[\C \setminus B_\delta](\varphi, e), \\
	I_K[\C \setminus B_\delta](u, \xi_2, \bar y) & \geq -I_K[\C \setminus B_\delta](\varphi, e) - C\osc(u),
	\end{align*}
	where the constant $C > 0$ is uniformly bounded above and below as $\sigma \to 0$.
	
	Again by the maximality of $(\bar x, \bar y)$, for every set $\mathcal O \subset \R^N$ such that $\bar x + z, \bar y + z \in \bar B_1$, we have
	\begin{equation*}
	I_K[\mathcal O](u, \tilde \varphi, x) - I_K[\mathcal O](u, \tilde \varphi, y) \leq I_K[\mathcal O](\psi, y) \leq C \Lambda \osc_{B_1}(u),
	\end{equation*}
	the latter inequality by the smoothness of $\psi$.
	
	Then, we replace these inequalities in~\eqref{Roma} to conclude that
	\begin{equation*}
	-C(a_0^\gamma |f|_\infty + \osc_{B_1} (u) + \|u\|_{L^1_\sigma} + 1) \leq 2 I_K[\C \setminus B_\delta](\varphi, e) + I_1,
	\end{equation*}
	and at this point we notice that the term $I_1 \to 0$ as $\delta \to 0$, meanwhile, by Dominated Convergence Theorem, we have $I_K[\C \setminus B_\delta](\varphi, e) \to I_K[\C](\varphi, e)$ as $\delta \to 0$.

	We make $\delta \to 0$, and take $\eta = c_1 |e|^{2\alpha}, \rho = c_1 |e|^\alpha$ for some constant $c_1 > 0$ universal. Using the estimates of Corollary 9 in~\cite{BCCI-Lip}, we have that
	\begin{equation}\label{crucial}
	\I[\C](\varphi, e) \leq - c L |e|^{1 - \sigma + \alpha (N + 2 - \sigma)},
	\end{equation}
	for some $c > 0$ is an universal constant. We stress on the fact that there exists $\epsilon_0 \in (0,1)$ just depending on $N$ such that $\epsilon_0 \leq c \leq \epsilon_0^{-1}$ for all $\sigma \in (1,2)$. 
	
	At this point, we fix $\alpha > 0$ small enough to get 
	$$
	-\theta := 1 - \sigma + \alpha(N + 2 - \sigma) \leq (1 - \sigma)/2 < 0.
	$$
	
	Thus, replacing the estimates above into~\eqref{Roma} we arrive at
	\begin{equation*}
	-C(|f|_\infty + \osc_{B_1} (u) + \|u\|_{L^1_\sigma} + 1) \leq -c L |e|^{-\theta}  
	\end{equation*}

    Recalling that $|e| \leq L^{-1} \osc_{B_1}(u)$,  taking $L \geq \osc(u)$ we have 
	\begin{equation*}
	-C(|f|_\infty + \osc_{B_1} (u) + \|u\|_{L^1_\sigma} + 1) \leq -c L,
	\end{equation*}
	and from here, we arrive at a contradiction with the choice of $L$ in~\eqref{choiceL}. This completes the proof.
%
\end{proof}

Now we present the result in the case $p$ in~\eqref{eqp} is small.
\begin{lema}\label{lemapsmall}
	Let $\sigma \in (1,2)$ and let $c_0$ be as in the previous lemma. If $|p| \leq c_0$, each bounded viscosity solution to~\eqref{eqp} is Lipschitz continuous, with Lipschitz constant independent of $p$.
\end{lema}

\begin{proof}
	The proof is esentially the same as the one of the previous lemma, so we will be sketchy.
	
	We consider the same functions $\varphi, \psi$ in Lemma~\ref{lemaplarge} and argue by contradiction. Thus, there exists a maximum point $(\bar x, \bar y) \in B_1 \times B_1$ for the function~\eqref{Phi}, with $\bar x \neq \bar y$.
	
	Then, denoting $\xi_1, \xi_2$ as in~\eqref{xis}, we can take $L$ large enough just in terms of  $\osc(u)$ and $c_0$ in order to have
	\begin{equation*}
	|p + \xi_1|, |p + \xi_2| \geq 1.
	\end{equation*}
	
	From this point, we follow the same proof above to conclude the result by fixing $L$ adequate.
	%
	%
	%
	%
\end{proof}

Now we are in position to provide the 

\medskip
\noindent
{\bf \textit{Proof of Theorem~\ref{first result}:}} The case $\sigma > 1$ is already proven in Proposition~\ref{propLipunif}. 

In the case $\sigma \in (0,1)$, we follow the same procedure as in Lemmas~\ref{lemapsmall}-\ref{lemaplarge} with $p=0$ and considering the function $\varphi(t) = Lt^\sigma$ in order to get the H\"older profile of the solution. The crucial estimate~\eqref{crucial} is obtained using the estimates of $T_2$ in the Step 3 of Theorem 1-$(i)$ in~\cite{BChI}. Similarly for the case $\sigma = 1$, where the same estimates are possible with the function $\varphi(t) = Lt^\alpha$ for any $\alpha \in (0,1)$.
\qed

\begin{remark}\label{rmkChJ}
	Some of our results in Theorem~\ref{first result} apply to other types of degenerate elliptic nonlocal operators. For instance, Chasseigne and Jakobsen in~\cite{ChJ} introduce and interesting nonlocal version of $p$-Laplace type operators with the form
\begin{equation*}
\Delta^{\sigma/2}_p u(x) = C_\sigma \int \limits_{\R^N} [u(x + j_p(Du(x))z) - u(x) - \mathbf{1}_B j_p(Du(x)) \cdot z] |z|^{-(N + \sigma)}dz,
\end{equation*}
for $p > 2$ and a jump function $j_p : \R^N \setminus \{ 0 \} \to \R^{N \times N}$ defined as
\begin{align*}
j_p(q) = |q|^{\frac{p-2}{2}} (I_N + r_p \hat q \otimes \hat q), 
\end{align*}
where $\hat q = q/|q|$ for $q \neq 0$ and $r_p > 0$ is a normalizing constant.

The method used in the proof of Theorem~\eqref{first result} (or more speciffically, of Lemma~\ref{lemapsmall}) applies to problems with the form
$$
\Delta^{\sigma/2}_p u = f \quad \mbox{in} \ B_1,
$$
in order to get that bounded solutions to this problem are $C^{\sigma}_{loc}(B_1)$ if $\sigma < 1$, and $C^{\alpha}_{loc}(B_1)$ for $\alpha < \min \{ 1, \sigma \}$ if $\sigma \geq 1$. Nevertheless, Lipschitz bounds and equicontinuity results in the sense of Proposition~\ref{propLipunif} are not direct to adapt due to the localization term driven by $\psi$ in~\eqref{Phi}, and this is why we do not pursue in this direction.

It is interesting to mention that $\Delta^{\sigma/2}_p \to \Delta_p$ as $\sigma \to 2^-$.
\end{remark}

\section{Proof of Theorem~\ref{second result}.}\label{teo2}

As it can be seen in~\cite{IS} (see Lemma 6 there), equation~\eqref{limIsaacs} has the particularity that the notion of viscosity solution can be equivalently defined if we require test functions do not have null gradient. Such a property is unavoidable linked to the homogeneity of the problem. This is a crucial property in the following stability result, since solutions to~\eqref{maineq} request test functions with no-null gradient.
\begin{lema}\label{approximation lemma}
For every $M>0$, $\epsilon, \alpha \in (0,1)$ and a modulus of continuity $\omega:[0,+\infty) \to \R_+$, there exist $\eta \in (0,1)$ and  $\sigma_0 \in (1 + \alpha, 2)$ such that, for each $p \in \R^N$ and a viscosity solution $u$ to 
$$
-\eta \leq -|Du+p|^{\gamma} \I_\sigma(u) \leq \eta \quad \mbox{in} \ B_1,
$$
satisfying the following conditions
\begin{itemize}	
\item[$(i)$] $\I_{\sigma}$ is a nonlocal operator as in~\eqref{operador} with $\sigma \in (\sigma_0, 2)$ and kernels $\{ K_{ij} \}$ satisfying~\eqref{contK} with respect to $\omega$,
\item[$(ii)$] $|u(x)-u(y)|\leq \omega(|x-y|) \quad \mbox{for} \ x,y \in \bar B_1$,
\item[$(iii)$] $|u(x)|\leq M(1 + |x|^{1+\alpha}) \quad \mbox{for } x \in \R^N$,
\end{itemize}
there exist a solution $h$ to problem~\eqref{limIsaacs} such that
$$
\sup_{B_{1}}|u-h|\leq \epsilon.
$$

\end{lema}
\begin{proof}
By contradiction, we assume there exists $M > 0$, $\epsilon, \alpha \in (0,1)$, a modulus of continuity $\omega$, and sequences $(p_k) \subset \R^N$, $\eta_k \to 0$, $\sigma_k \to 2$, and a family of functions $\{ u_k \}$ such that $(ii), (iii)$ above holds with $u = u_k$ and solving, 
in the viscosity sense, the problem
%
$$
-\eta_k \leq -|D u_k + p_k|^{\gamma} \I_{\sigma_k}(u_k) \leq \eta_k \quad \mbox{in} \ B_1,
$$
and such that
\begin{equation}\label{contra}
\sup_{B_{1}}|u_k-h|> \epsilon,
\end{equation}
where $h$ is a solution to~\eqref{limIsaacs}.

\medskip

However, the uniform continuity condition $(ii)$ together with the uniform integrability condition $(iii)$ allows us to pass to the limit and, up to subsequences, we have $u_k$ converge locally uniformly in $B_1$ to a solution to~\eqref{limIsaacs}. This is a contradiction to~\eqref{contra}.
%
%
\end{proof}


The last challenge is iterate the Lemma \ref{approximation lemma} and be able to guarantee that there exists a sequence $l_k$ of first order polynomials that approximate $u$ solution to~\eqref{maineq} in a suitable way, and at different scales.

\begin{prop}\label{iterative teo}
Let $\gamma > 0$. There exist $\rho, \alpha, \eta \in (0,1)$, $\sigma_0 \in (1,2)$ and $C > 0$ such that, for each continuous viscosity solution $u$ to the problem 
\begin{align}\label{aracaju1}
-\eta \leq -|Du|^\gamma \I_{\sigma}(u) \leq \eta \quad \mbox{in} \ B_1, 
\end{align}
with $\|u\|_{L^\infty(\R^N)} \leq 1$ and $\sigma \in (\sigma_0, 2)$, there exist a sequence $l_k=a_k+p_kx$ such that
\begin{equation}\label{aracaju2}
\sup_{B_{\rho^k}}|u-l_k|\leq \rho^{k(1+\alpha)},
\end{equation}
such that
\begin{equation}\label{estimatescoeficientpoly}
\left\lbrace\begin{array}{ll}
	|a_{k+1}-a_k|\leq & C\rho^{(1+\alpha)k}\\ 
	|b_{k+1}-b_k|\leq & C\rho^{\alpha k}. 
\end{array}\right. 
\end{equation}	
\end{prop}
\begin{proof}
	In what follows, we consider universal constants $\bar A > 0, \bar \alpha \in (0,1)$ such that, for all function $h \in C(B_1)$ with $\| h \|_{L^\infty(B_1)} \leq 4$, and satisfying~\eqref{limIsaacs}, then 
	\begin{equation*}
	|h(0)|, |Dh(0)| \leq \bar A; \quad |h(x) - h(0) - Dh(0)x| \leq \bar A |x|^{1 + \bar \alpha}, \quad x \in B_1.
	\end{equation*}
	
	This is possible by classical interior $C^{1, \bar \alpha}$ estimates, see~\cite{CC}.
	
\medskip

By a slight modification of $u$, we can assume $u(0) = 0$.

We are going to fix $\alpha \in (0, \bar \alpha)$ such that $\sigma  - 1 - \alpha (1 + \gamma) > 0$, and $\rho \in (0,1)$ small enough in order to have
\begin{equation}\label{rho}
\rho^{\bar \alpha - \alpha}(1 +  \bar A) \leq \frac{1}{100}.
\end{equation}

Taken in such a way, we are going to pick $\eta$ small enough in order and $\sigma_0$ sufficiently close to $2$ (this does not affect the choice of $\alpha$) such that for each solution $u$ to~\eqref{aracaju1} satisfying the growth condition $|u(x)| \leq 1 + |x|^{1 + \bar \alpha}, \ x \in \R^N$ we have the existence of a function $h$ solving~\eqref{limIsaacs} such that $\| u - h \|_{L^\infty(B_1)} \leq \rho^{1 + \alpha}/2$. This is possible in view of Lemma~\ref{approximation lemma} since the equicontinuity property $(ii)$ is fullfilled by the regularity results proven in Theorem~\ref{first result}.
	

We denote $l_0 = 0$, and we are going to construct inductively a sequence of linear functions $l_k$ with the form $l_k x = a_k + p_k x$ where $a_k \in \R, \ p_k \in \R^N$ and a sequence of functions $\{w_k\}_{k \geq 0}$ defined as
\begin{equation}\label{wk+1}
w_{k+1}(x) = \frac{u(\rho^{k}x) - l_k(\rho^{k}x)}{\rho^{k(1 + \alpha)}}, \quad x \in \R^N.
\end{equation}

It is easy to see that $w_{k}$ solves
\begin{equation*}
- \rho^{\sigma-\alpha\gamma-1-\alpha}\eta \leq -|D w_k + p_k|^{\gamma}\I(w_k) \leq \rho^{\sigma-\alpha\gamma-1-\alpha} \eta, \quad \mbox{in} \ B_{\rho^{-{k-1}}},
\end{equation*}
and therefore, taking $\alpha$ small enough and $\rho < 1/4$, in particular we have
\begin{equation*}
- \eta \leq -|D w_k + p_k|^{\gamma}\I(w_k) \leq \eta, \quad \mbox{in} \ B_{1}.
\end{equation*}

We construct the sequence $l_k$ as follows: once it is proven that 
\begin{equation}\label{aracaju}
|w_k(x)| \leq 1 + |x|^{1 + \bar \alpha}, \quad x \in \R^N,
\end{equation}
then by the choice of the parameters above, we have the existence of a function $h_k \in C^{1, \bar \alpha}$ solving~\eqref{limIsaacs} such that $\| w_k - h_k\|_{L^\infty(B_1)} \leq \rho^{1 + \alpha}/2$.

Now we explain how to define $a_k, p_k$ in order to have~\eqref{aracaju} inductively. Once $w_k$ is constructed and satisfies~\eqref{aracaju}, we consider $h_k$ as above and define $\tilde a_k = h_k(0)$ and $\tilde p_k = D h_k(0)$, we write $\tilde l_k(x) = \tilde a_k + \tilde p_k x$ and define
\begin{equation}\label{lk}
l_{k + 1}(x) = l_k(x) + \rho^{k(1 + \alpha)} \tilde l_k(\rho^{-k}x), \quad x \in \R^N,
\end{equation}
or equivalently
$$
a_{k + 1} = a_k + \rho^{1 + \alpha} \tilde a_k; \quad p_{k + 1} = p_k + \rho^{\alpha} \tilde p_k.
$$

Starting with $w_0 = u$ and $l_0 = 0$, in view of definition~\eqref{wk+1}, the sequences $\{ w_k \}, \{ l_k \}$ are well-defined. Once~\eqref{aracaju} is proven,~\eqref{aracaju2} holds directly, and~\eqref{estimatescoeficientpoly} holds from~\eqref{lk} with $C = \bar A$.

\medskip

Thus, the rest of the proof is focused on an inductive proof for~\eqref{aracaju}, assuming it holds for $k$ and proving it for $k + 1$. For this, it is direct to check that
$$
w_{k + 1}(x) = \frac{w_k(\rho x) - \tilde l_k (\rho x)}{\rho^{1 + \alpha}}.
$$ 

Thus, we see that for $|x| \rho \geq 1/2$ we have
\begin{align*}
|w_{k + 1}(x)| \leq & \rho^{-(1 + \alpha)} (|w_k(\rho x)| + |\tilde l_k(\rho x)|) \\
\leq & \rho^{-(1 + \alpha)} (1 + \rho^{1 + \bar \alpha}|x|^{1 + \bar \alpha}) + \rho^{-(1 + \alpha)}\bar A (1 + \rho |x|) \\
\leq & \rho^{\bar \alpha - \alpha} (5 + 6 \bar A) |x|^{1 + \bar \alpha}.
\end{align*}

Then, by the choice of $\rho$ in~\eqref{rho}, we conclude $|w_{k + 1}(x)| \leq |x|^{1 + \bar \alpha}$.

\medskip

In the case $|x| \rho \leq 1/2$, we see that
\begin{align*}
|w_{k + 1}(x)| \leq & \rho^{-(1 + \alpha)} (|w_k(\rho x) - h_k(\rho x)| + |h_k(\rho x) - \tilde l_k(\rho x)|) \\
\leq & \rho^{-(1 + \alpha)}(\rho^{1 + \alpha}/2 + \bar A |\rho x|^{1 + \bar \alpha}) \\
\leq & 1/2 + \bar A \rho^{\bar \alpha - \alpha}|x|^{1 + \bar \alpha},
\end{align*}
and again by~\eqref{rho}, we conclude the result.
\end{proof}

\medskip

 Now we are in position to provide the
	
\noindent
{\bf \textit{Proof of Theorem~\ref{second result}:}}

First we consider 
$$
	\bar{u}=\frac{u}{\|u\|_\infty+(\eta^{-1}\|f\|_\infty^{\frac{\sigma-1}{1+\gamma}})}
$$
so $\|\bar{u}\|_\infty\leq 1$ and $\|f\|_\infty\leq \eta$ attending the assumptions of the Proposition \ref{iterative teo}. Now we will prove that there exists a first order polynomial $l_\infty=a_\infty+p_\infty x$ such that

$$
	|\bar{u}(x)-l_\infty(x)|\leq C_0|x|^{1+\alpha}.
$$
By (\ref{estimatescoeficientpoly}) we have that $a_k\in \R$ and $p_k\in \R^N$ are Cauchy sequences then $a_k\rightarrow a_\infty$ and $p_k\rightarrow p_\infty$. Moreover

$$
|l_\infty-l_k|\leq \sum_k^{\infty}|l_{k+1}-l_k|\leq C\rho^{(1+\alpha)k}.
$$
Finally we take $|x|\leq 1$ and $k$ such that $\rho^{k+1}\leq |x| < \rho^k$, so

$$
	|\bar{u}-l_\infty|\leq |l_\infty-l_k|+|u-l_k|\leq (1+C)\rho^{(1+\alpha)k},
$$
then
$$
	|\bar{u}-l_\infty|\leq \frac{(1+C)}{\rho^{1+\alpha}}|x|^{1+\alpha}.
$$
By the Campanato theory the above estimate assure us that

$$
[\bar{u}]_{C^{1+\alpha}(B_{1/2})}\leq C_0
$$
therefore
$$
[u]_{C^{1+\alpha}(B_{1/2})}\leq C_0(\|u\|_\infty+\|f\|_\infty^{\frac{\sigma-1}{1+\gamma}}).
$$

\qed

\bigskip

\noindent
\thanks{\textbf{Aknowledgements:} 
DP  was partially supported by Capes-Fapitec and CNPq. E.T. was partially supported by Conicyt PIA Grant No. 79150056, Foncedyt Iniciaci\'on Grant No. 11160817. Both authors are thankful to Departmento de Matem\'atica de UFS, and DMCC-Usach for the hospitality of their respective visits to these centers, which were also supported by Promob/Capes-Fapitec, Apoyo a Asistencia de Eventos Dicyt-Usach, Fondecyt and CNPq.


\end{document}